\documentclass[12pt]{amsart}
\usepackage{amsthm,amsmath,a4,amssymb}
\usepackage[overload]{empheq} 

\newtheorem{thm}{Theorem}[section]
\newtheorem{cor}[thm]{Corollary}
\newtheorem{lem}[thm]{Lemma}
\newtheorem{prop}[thm]{Proposition}
\theoremstyle{definition}
\newtheorem{defn}[thm]{Definition}

\newtheorem{exe}[thm]{Example}
\numberwithin{equation}{section}

\newcommand{\Hom}{\textnormal{Hom}}

\newcommand{\GL}{\textnormal{GL}}
\newcommand{\topdim}{\textnormal{tdim}}
\newcommand{\inddim}{\textnormal{indim}}

\newcommand{\Q}{\mathbf{Q}}
\newcommand{\Z}{\mathbf{Z}}
\newcommand{\R}{\mathbf{R}}
\newcommand{\eps}{\varepsilon}

\begin{document}

\address{IRMAR, Campus de Beaulieu, 35042 Rennes CEDEX, France}
\email{yves.decornulier@univ-rennes1.fr}

\title[On the Chabauty space of LCA groups]{On the Chabauty space of locally compact abelian groups}
\author{Yves Cornulier}%
\date{November 24, 2010}



\begin{abstract}
This paper contains several results about the Chabauty space of a general locally compact abelian group. Notably, we determine its topological dimension, we characterize when it is totally disconnected or connected; we characterize isolated points.
\end{abstract}
\maketitle
\section{Introduction}

Let $X$ be a locally compact Hausdorff space. The set $\mathcal{F}(X)$ of closed subsets can be endowed with the {\it Chabauty topology}, which makes it a compact Hausdorff space. For this topology, a net $(Y_i)$ converges to $Y$ if and only if $(Y_i\cup\{\infty\})$ converges to $Y\cup\{\infty\}$ in the Hausdorff topology of the one-point compactification of $X$. When $X$ is second countable, $\mathcal{F}(X)$ is metrizable. See details in Paragraph \ref{chab}.
If $G$ is a locally compact group, the set $\mathcal{S}(G)$ of closed subgroups of $G$ is closed in $\mathcal{F}(G)$ and therefore is compact Hausdorff as well.

Introduced by Chabauty in \cite{Cha}, the Chabauty topology has been studied in \cite{HP,PTd,PT,PT2,FG1,FG2,BHK,Klo,Ha}, and more specifically for discrete groups in \cite{Gri,Chm,CGP,CGPab}. The fine study of $\mathcal{S}(G)$ is subtle even for apparently simple examples. While it is readily seen that $\mathcal{S}(\R)$ is homeomorphic to a segment, a tricky argument due to Hubbard and Pourezza \cite{HP} shows that $\mathcal{S}(\R^2)$ is homeomorphic to the 4-sphere. For $n\ge 3$, $\mathcal{S}(\R^n)$ is known to be singular (i.e.~not a topological manifold even with boundary) but Kloeckner \cite{Klo} showed that $\mathcal{S}(\R^n)$ has a natural ``stratification" which in particular makes it a simply connected and locally contractible space; however $\mathcal{S}(\R^n)$ has not yet unveiled all its mysteries and for instance it is still unknown whether it can be triangulated. Besides, Haettel \cite{Ha} gave a full description of the space $\mathcal{S}(\R\times\Z)$, showing in particular that it is path-connected but not locally connected, and has uncountable fundamental group.

Our first result, which underlies the proof of all others, is the continuity of the orthogonal map in Pontryagin duality. Let $G$ be a locally compact abelian group. Under the topology of uniform convergence, the abelian group $G^\vee=\Hom(G,\R/\Z)$ is locally compact, and the main result in Pontryagin duality is that the natural homomorphism $G\to (G^\vee)^\vee$ is a topological group isomorphism. In Bourbaki \cite[Chap.~II.2]{Bou}, Pontryagin duality is used to deduce fundamental results in the structure theory on locally compact abelian groups. One of the results \cite[Chap.~II.1, no.~7]{Bou} is that the orthogonal map \begin{align*}\mathcal{S}(G) & \;\;\to\;\;  \mathcal{S}(G^\vee)\\ H &  \;\;\mapsto\;\;  H^\curlywedge=\{\phi\in\Hom(G,\R/\Z):\phi(H)=0\}\end{align*} is a bijection.

\begin{thm}\label{pc}
Let $A$ be an abelian group and $\hat{A}=\Hom(A,\R/\Z)$ its Pontryagin dual. Then the orthogonal map $\mathcal{S}(A)\to\mathcal{S}(\hat{A})$, $H\to H^\curlywedge$, is a homeomorphism.
\end{thm}

We refer to this result as {\it Pontryagin-Chabauty duality}, and we develop consequences on the general structure of the space $\mathcal{S}(G)$. The next theorem deals with the topological (or covering) dimension, which is a (possibly infinite) integer number $\topdim(X)$ associated to any topological space $X$, invariant under homeomorphism, and for which $\topdim(\R^n)=n$ for all~$n$. See Paragraph \ref{todi} for details.

\begin{thm}[Theorem \ref{covdim}]
If $G$ is any locally compact abelian group, then the topological (covering) dimension of $\mathcal{S}(G)$ is given by
$$\topdim(\mathcal{S}(G))=\topdim(G)\topdim(G^\vee),$$
where $0\infty=\infty 0=0$. In particular, if $G=\R^d\times\Z^\ell\times\R/\Z^{m}$ then $$\topdim(\mathcal{S}(G))=(d+\ell)(d+m).$$\label{covdimi}
\end{thm}
Several characterizations of $\topdim(G)$ and $\topdim(G^\vee)$ are recalled in Paragraph \ref{dimlca}. Theorem \ref{covdimi} is based on the non-trivial special case $$\topdim(\mathcal{S}(\R^d))=d^2,$$
which follows from Kloeckner's local description \cite{Klo}, see 
Section \ref{sec:Dimension}.

If $G$ is a locally compact abelian group, from classical theory it can be written as $\R^k\times M$ so that $M$ has a compact open subgroup; the finite number $R(G)=k$ is uniquely defined (see Paragraph \ref{invr}). The following results bring out a dichotomy between the case $R(G)=0$ (i.e.\ $G$ is compact-by-discrete) and the case $R(G)\ge 1$.

\begin{thm}[Section \ref{seco}]
Let $G$ be a locally compact abelian group with $R(G)\ge 1$. Then $\mathcal{S}(G)$ is connected. If moreover $G$ is a Lie group and is compactly generated (i.e.\ $G/G_0$ is finitely generated, where $G_0$ is the unit component), then $\mathcal{S}(G)$ is path-connected.\label{r1i} 
\end{thm}

Nevertheless, Proposition \ref{sra} exhibits countable discrete abelian groups $A$ such that $\mathcal{S}(\R\times A)$ is not path-connected.

\begin{exe}
By Theorems \ref{covdimi} and \ref{r1i}, if $G$ is a locally compact abelian group, then $\mathcal{S}(G)$ is both connected and one-dimensional if and only if $G\simeq \R\times H$, where $H$ is profinite-by-(discrete torsion). In many cases, like $\R\times\Q_p$, it also follows from Theorem \ref{con} that $\mathcal{S}(G)$ is path-connected. It would be interesting to have a closer look into $\mathcal{S}(G)$ for those examples. 
\end{exe}

If $X$ is a totally disconnected compact Hausdorff space, we define $\pi_0(X)$ as the quotient of $X$ by its partition by connected components. By \cite[III.4.4, Proposition~7]{BourbakiT}, $\pi_0(X)$ is compact, Hausdorff and totally disconnected. Section \ref{sec:rzero} studies locally compact groups $G$ with $R(G)=0$ and more precisely the connected components of $\mathcal{S}(G)$. In particular we get

\begin{thm}\label{pizi}
Let $G$ be a locally compact abelian group with $R(G)=0$. Then
\begin{itemize}
\item $\pi_0(\mathcal{S}(G))$ is infinite if and only if $G$ is infinite;
\item every connected component of $\mathcal{S}(G)$ is homeomorphic to a compact group; if $G$ is a compactly generated Lie group, these components are tori.
\item $\mathcal{S}(G)$ is totally disconnected if and only if $G$ is either totally disconnected or elliptic (i.e.~is the union of its compact subgroups). 
\end{itemize}
\end{thm}

Also, Theorem \ref{pizero} gives a structure result for locally compact abelian groups for which $\mathcal{S}(G))$ has countably many components countable. 
We extract from it

\begin{thm}\label{pizeroi}
Let $G$ be a locally compact abelian group which is neither discrete nor compact. Then $\mathcal{S}(G)$ has (at most) countably many connected components if and only if one of the following condition holds
\begin{itemize}
\item $G=\Q_\ell\times\Z_m\times \mathbf{C}_{n^\infty}\times F$, where $F$ is finite, $\ell,m,n\ge 1$, $\ell$ is prime with $mn$;
\item $G=K\times D\times F$ with $D$ torsion-free discrete, $K$ compact connected, $F$ finite, and $\mathcal{S}(D)$ and $\mathcal{S}(K)$ are countable;
\item $R(G)\ge 1$.
\end{itemize}
\end{thm}
Note that in the first case $G$ is both elliptic and totally disconnected so $\mathcal{S}(G)$ is countable itself, while in the second case $\mathcal{S}(G)$ is not totally disconnected by Theorem \ref{pizi}. Note that the case when $G$ is discrete (characterization of abelian groups with countably many subgroups) is done in \cite{Boy} (see Proposition \ref{boyy}) and the compact case follows by Pontryagin duality. Note that the proof of Theorem \ref{pizeroi} involves an intrinsic characterization of finite direct products of $p$-adic groups $\Q_p$ (see Lemma \ref{caradic}).

The next result concerns isolated points. When $G$ is a discrete abelian group, a necessary and sufficient condition for a subgroup $H$ to be isolated in $\mathcal{S}(G)$ is that $H$ is isolated in $\mathcal{S}(H)$ and $\{1\}$ is isolated in $\mathcal{S}(G/H)$ (see for insytance \cite{CGPab}). However when $G$ is not assumed discrete, this condition obviously remains necessary but is not sufficient: for instance $\Z$ is not isolated in $\mathcal{S}(\R)$. The characterization of isolated points goes as follows (we refer to Paragraph \ref{subcla} for the definition of Artinian and adic groups)

\begin{thm}\label{posil}
Let $G$ be a locally compact abelian group and $H$ a closed subgroup.
Then $H$ is an isolated point in $\mathcal{S}(G)$ if and only if we are in one of the two (dual to each other) cases
\begin{enumerate}
\item\label{tddddd} $H\simeq A\times P$, $G/H\simeq D$, $A$ finitely generated abelian, $P$ adic, $D$ Artinian (so $G$ is totally disconnected);
\item $H\simeq P$, $G/H\simeq T\times D$, $P$ adic, $T$ torus, $D$ Artinian (so $G$ is elliptic).
\end{enumerate}
\end{thm}

In \cite{CGPab}, the study was pursued to a Cantor-Bendixson analysis of $\mathcal{S}(G)$, when $G$ is a countable discrete abelian group, leading to the determination of the homeomorphism type of $\mathcal{S}(G)$.
It would be interesting to generalize this to second countable, totally disconnected locally compact abelian groups. Among those groups, those for which this question is nontrivial have a very special form: they have to lie in 
an exact sequence $$0\to A\times P\to G\to D\to 0$$ with $A$, $P$, $D$ as in Theorem \ref{posil}(\ref{tddddd}), since otherwise if $G$ is not of this form, by the same theorem, $\mathcal{S}(G)$ has no isolated point and therefore is a Cantor space.

\tableofcontents

\noindent {\bf Acknowledgements.} I thank P.~de la Harpe, B.~Kloeckner and R.~Tessera for useful comments.

\section{Preliminaries}

This section recalls definitions and basic results used throughout the paper.


\subsection{The Chabauty topology}\label{chab}

Recall that if $X$ is a locally compact space, we denote by $\mathcal{F}(X)$ the set of closed subsets of $X$. This set has several natural topologies. 
 The one of interest for us is called the {\it compact topology}, and is compact Hausdorff. It appears in an exercise by Bourbaki \cite[Chap.~8, \S 5]{BouI} and was reintroduced by Narens \cite{Nar} in the context and language of non-standard analysis and its description in standard terms was then provided by Wattenberg \cite{Wat}. Given a compact subset $K\subset X$ and open subsets $U_1,\dots,U_k\subset X$, define
$$\Omega(K;U_1,\dots,U_k)=\{F\in\mathcal{F}(X):F\cap K=\emptyset;\forall i,F\cap U_i\neq\emptyset\};$$
these set form the basis of the compact topology on $\mathcal{F}(X)$.

The reader can prove as a simple exercise the following characterization of converging nets in $\mathcal{F}(X)$.
\begin{lem}For a net $(F_i)$ of closed subsets of $X$ and $F\in\mathcal{F}(X)$, we have equivalences
\begin{itemize}
\item $F_i\to F$ in the compact topology;
\item for every compact subset $K$ and any open subsets $U_1,\dots,U_k$ with $F\cap K=\emptyset$ and $F\cap U_j\neq\emptyset$ for each $j$, we have eventually $F_i\cap K=\emptyset$ and $F_i\cap U_j\neq\emptyset$ for each $j$.
\item For every $x\in F$ and every neighbourhood $V$ of $x$, we eventually have $F_i\cap V\neq\emptyset$, and for every $x\in X-F$ there exists a neighbourhood $V$ of $x$ such that eventually $F_i\cap V=\emptyset$.
\end{itemize}
\end{lem}

From the latter characterization it is straightforward that if $X=G$ is a locally compact group, then $\mathcal{S}(G)$ is closed in $\mathcal{F}(X)$; in this specific case and under countability assumptions, this topology was introduced by Chabauty \cite{Cha}.

\subsection{Dimension of topological spaces}\label{todi}

We recall briefly several notions of dimension. For details see \cite{HW}. Let $X$ be a topological space. If $\mathcal{U}=(U_i)_{i\in I}$ an open covering of $X$, define its {\it degree} as $$\textnormal{deg}(\mathcal{U})=\sup\left\{n:\;\exists J\subset I, \;\#J=n,\;\bigcap_{i\in J} U_i\neq\emptyset\right\},$$
and the {\it topological dimension} (or {\it covering dimension}) $\topdim(X)$ of $X$ is defined as the smallest $n$ such that every open covering of $X$ can be refined to an open covering with degree $\le n+1$ (by convention $\topdim(\emptyset)=-1$). This dimension is well-behaved in many respects, for instance $\topdim(\R^n)=n$.

The inductive dimension $\inddim(X)$ of a topological space $X$ is defined inductively as follows: $\inddim(X)=-1$ if and only if $X=\emptyset$; otherwise $\inddim(X)\le n$ if and only if every $x\in X$ has a basis of closed neighbourhoods $(V_i)$ such that for each $i$ the boundary of $V_i$ has inductive dimension $\le n-1$. 

These dimensions are related. Let us state, for later reference
\begin{prop}\label{toporef}~
\begin{enumerate}
\item\label{ury} By a theorem of Urysohn (see \cite{HW}), if $X$ is a separable metrizable space, then $\inddim(X)=\topdim(X)$. This applies in particular to $X=\mathcal{S}(G)$ when $G$ is a second countable locally compact group.
\item\label{lcdim} By a theorem of Pasynkov \cite{Pas}, if $G$ is a locally compact group, then $\inddim(X)=\topdim(X)$.
\item\label{alek} A theorem of Aleksandrov (see \cite{Isb}) states that if is any compact Hausdorff space then $\topdim(X)\le\inddim(X)$; this is an equality when $X$ is metrizable by (\ref{ury}) but not for general compact Hausdorff spaces \cite{Vop}.
\item\label{invers} The topological dimension of an inverse limit of Hausdorff compact spaces of topological dimension $\le k$ is also $\le k$ \cite{Kak}.
\end{enumerate}
\end{prop}

\subsection{Subclasses of groups}\label{subcla}

Recall that, for $n\ge 2$, the Pr\"ufer group $\mathbf{C}_{n^\infty}$ is defined as the inductive limit of the groups $\mathbf{C}_{n^k}=\Z/n^k\Z$; in particular $\mathbf{C}_{n^\infty}\simeq\Z[1/n]/\Z$ and is the direct product of $\mathbf{C}_{p^k}$ where $p$ ranges over {\it distinct} prime divisors of $n$. Similarly, the ring $\Z_n$ of $n$-adic numbers is defined as the projective limit of the groups $\Z/n^k\Z$ and is the product of $\Z_p$ where $p$ ranges over {\it distinct} prime divisors of $n$. Also, $\Q_n$ denotes the product of $p$-adic fields $\Q_p$ when $p$ ranges over {\it distinct} divisor of $n$, and we call a group {\it local} if it is isomorphic to a finite direct product of $\Q_p$ (distinct or not).

Recall that a discrete abelian group is {\it artinian} if it satisfies the descending condition on subgroups (no infinite decreasing chain), or equivalently is a finite direct product of Pr\"ufer groups and finite groups. We say here that a locally compact group is {\it adic} if it is a finite direct product of finite groups and groups of the form $\Z_n$ for some (non-fixed) $n$.

We call {\it torus} a group of the form $\R^k/\Z^k$.


\subsection{On Pontryagin duality}\label{opd}

The following tables shows various groups or classes of locally compact abelian groups, in correspondence by Pontryagin duality. See \cite[Chap.~II.2]{Bou} for details.

 \makebox[\linewidth]{
\small
\begin{tabular}{|c|c|}
\hline
$G$ & $G^\vee$\\ 
\hline
$\R$ & $\R$\\ 
\hline
$\Q_p$ & $\Q_p$ \\
\hline
$\Z/n\Z$ & $\Z/n\Z$ \\
\hline
$\Z$ & $\R/\Z$ \\
\hline
$\Z_p$ & $\mathbf{C}_{p^\infty}$ \\
\hline
\end{tabular}
}

\medskip

 \makebox[\linewidth]{
\small
\begin{tabular}{|c|c|}
\hline
{\it class} $\mathcal{C}$ & $\mathcal{C}^\vee$\\ 
\hline
discrete & compact\\ 
\hline
Lie & compactly generated \\
\hline
totally disconnected & elliptic \\
\hline
connected & torsion-free \\
\hline
discrete torsion & profinite \\
\hline
discrete artinian & adic \\
\hline
discrete finitely generated & torus-by-finite \\
\hline
\end{tabular}
}

\subsection{The invariant $R$}\label{invr}

Let $G$ be a locally compact abelian group. Define
$$R(G)=\sup\{k|\;\R^k\text{ is isomorphic to a direct factor of }G\}.$$

The following result is contained in \cite[II.2, Proposition~3(i)]{Bou}
\begin{prop}\label{rcd}
Every locally compact abelian group $G$ is isomorphic to the direct product of $\R^k$ and a compact-by-discrete group.

If $G$ is a compactly generated Lie group, then it isomorphic to $\R^k\times\Z^\ell\times (\R/\Z)^m\times F$ with $F$ finite and $k,\ell,m$ non-negative integers.
\end{prop}

From this, we immediately derive the following results.

\begin{lem}~\label{Rop}
\begin{itemize}
\item If $G_1$ is an open subgroup of $G_2$, then $R(G_1)=R(G_2)$.
\item If $K$ is a compact subgroup of $G$ then $R(G/K)=R(G)$.
\item For every locally compact abelian group, $R(G)<\infty$.
\item We have $R(G)=0$ if and only if $G$ is compact-by-discrete.
\end{itemize}
\end{lem}

\subsection{On connected groups}


\begin{lem}\label{ccntp}
If a connected compact locally compact abelian group $G$ is non-trivial, then it contains a nontrivial path.
\end{lem}
\begin{proof}
As $G^\vee$ is a non-trivial discrete torsion-free abelian group (see Paragraph \ref{opd}), we have $\Hom(G^\vee,\R)\neq 0$. By Pontryagin duality, $\Hom(\R,G)$ is non-trivial as well. In particular, $G$ contains non-trivial paths.
\end{proof}

\begin{lem}\label{hnz}
Let $D$ be a torsion-free discrete abelian group and $Q$ a connected compact abelian group. If both $D$ and $Q$ are non-zero, then $\Hom(D,Q)$ is a non-trivial compact connected group.
\end{lem}
\begin{proof}
We have a non-zero element $f$ in $\Hom(D,\R)$ and by Pontryagin duality, we can find a non-zero $g$ in $\Hom(\R,Q)$ as well. By composing $g$ by $\lambda f$ for suitable $\lambda>0$, we get a non-trivial element in $\Hom(D,Q)$. 

As $D$ is a direct limit of groups of the form $\Z^k$, $\Hom(D,Q)$ is a projective limit of groups of the form $\Hom(\Z^k,Q)=Q^k$ which are compact and connected. Therefore $\Hom(D,Q)$ is compact and connected as well.
\end{proof}

\subsection{Dimension of locally compact abelian groups}\label{dimlca}

We deal here with the topological dimension because we need to use Proposition \ref{toporef}(\ref{invers}), but by (\ref{lcdim}) of the same proposition, it coincides for locally compact groups with the inductive dimension.
The following lemma is a particular case of \cite[Theorem~5]{Dix}. Recall that all homomorphisms are assumed to be continuous.

\begin{lem}
If $G$ is a locally compact abelian group, $H$ a closed subgroup and $f:H\to\R$ a homomorphism, then $f$ can be extended to a homomorphism $G\to\R$.\label{rex}
\end{lem}
\begin{proof}
Let $K$ be a compact subgroup of $G$ such that $G/K$ is a Lie group whose unit component is isomorphic to $\R^k$. Working in $G/K$ if necessary, we can suppose $K=0$. Let $H_1$ be the inverse image in $G$ of the torsion in $G/H$. Then $f$ has a unique extension to $H_1$; in restriction to $H_1\cap G_0$ this extension is continuous as it essentially consists in extending a homomorphism from $\R^\ell\times\Z^m$ to $\R^{\ell +m}$. So it is continuous. Now there exists a direct factor of $H_1\cap G_0$ in $G_0$, so $f$ extends to $H+G_0$, and finally by injectivity of the $\Z$-module $\R$, we can extend $f$ to all of $G$.
\end{proof}

\begin{lem}The following (possibly infinite) numbers are equal
\begin{itemize}
\item The supremum of $k$ such that there exists a homomorphism $G\to\R^k$ such that the closure of the image is cocompact
\item The supremum of $k$ such that $\Z^k$ embeds discretely into $G$.
\end{itemize}\label{dimgvv} 
\end{lem}
\begin{proof}
Suppose that $G\to\R^k$ is a homomorphism whose image has cocompact closure. Then the image contains a basis, hence a lattice, which lifts to a discrete subgroup of $G$ isomorphic to $\Z^k$.

Suppose that $\Z^k$ embeds discretely into $G$. Consider the embedding of $\Z^k$ in $\R^k$ as a lattice. By Lemma \ref{rex}, this can be extended to a homomorphism $G\to\R^k$ whose image has cocompact closure.
\end{proof}

\begin{lem}The following (possibly infinite) numbers are equal
\begin{itemize}
\item The topological dimension $\topdim(G)$;
\item The supremum of $k$ such that there exists a homomorphism $\R^k\to G$ with discrete kernel;
\item The supremum of $k$ such that $(\R/\Z)^k$ is a quotient of $G$.
\end{itemize}\label{covg} 
\end{lem}
\begin{proof}
Denote by $a,b,c$ the corresponding numbers. By Lemma \ref{dimgvv} and Pontryagin duality, $b=c$. Clearly $a\ge b$.

Let us prove $a\le c$. The three numbers are invariant if we replace $G$ by an open subgroup, so we suppose $G/G_0$ compact. Therefore we can write $G=\R^\ell\times K$ with $K$ compact, and we write $K$ as a filtering projective limit of compact Lie groups $K_i$. If $c=c(G)=k+\ell<\infty$, then $c(K_i)\le k$ for all $i$, so $K_i$ is a compact Lie group of dimension at most $k$, hence of topological dimension at most $k$. By Proposition \ref{toporef}(\ref{invers}), $K$ has dimension at most $k$, so $G$ has dimension at most $k+\ell$. 
\end{proof}

By Pontryagin duality again, we get

\begin{cor}\label{dimgvc}
The numbers in Lemma \ref{dimgvv} are also equal to the topological dimension of $G^\vee$.
\end{cor}

We will also need the following lemma.

\begin{lem}\label{ds}
Let $G$ be a locally compact abelian group and $k=R(G)$. There exist abstract sets $I,J$, a compact subgroup $K$ in $G$ and an open subgroup $H$ containing $K$, such that
\begin{itemize}\item $H/K$ is isomorphic to $\Z^{(I)}\times (\R/\Z)^J\times\R^k$, \item $\dim(G)=\dim(H/K)=k+\# I$, and\item $\dim(G^\vee)=\dim((H/L)^\vee)=k+\# J$.\end{itemize} (Here the cardinal of an infinite set is just $\infty$.)
\end{lem}
\begin{proof}
In view of Proposition \ref{rcd}, we can suppose $G$ has a compact open subgroup $M$. Consider a maximal free subgroup $H/M$ in the discrete group $G/M$. Obviously $\dim(G)=\dim(H)$, and since $G/M$ is torsion, it follows from Corollary \ref{dimgvc} that $\dim(G^\vee)=\dim(H^\vee)$. Apply this to find, by duality, a closed group $K$ of $M$ such that $(M/K)^\vee$ is maximal free abelian, so $\dim(H)=\dim(H/K)$ and $\dim(H^\vee)=\dim((H/K)^\vee)$. Now $M/K$ is connected, hence divisible, and open in $H/K$, so has a direct factor. So $H/K$ is isomorphic to $\Z^{(I)}\times (\R/\Z)^J$ for some sets $I,J$.
\end{proof}

\section{Pontryagin-Chabauty duality}

Theorem \ref{pc} is stated in \cite{PTd}, and the short and elementary proof given therein consists in a reduction to the case of a Euclidean space, but the latter case is considered there as ``easily verified". Also, continuity of the orthogonal map in $\mathcal{S}(\R^d)$ is asserted in \cite[Section~2.4]{Klo}. We here give a detailed proof of this not-so-obvious fact.
 
\begin{thm}\label{pce}
Let $V$ be a finite-dimensional real vector space. The orthogonal map $\mathcal{S}(V)\to\mathcal{S}(V^\vee)$, is a homeomorphism.
\end{thm}

We need a few preliminary results.

\begin{lem}\label{matrices}
Let $W$ be a closed subgroup of $V$ and $\Gamma$ a lattice in $W$. Let $(W_n)$ be a sequence of closed subgroups of $V$ such that $W_n\to W$. Then there exists $A_n\in\GL(V)$ with $A_n\to 1$ and $\Gamma\subset A_nW_n$.
\end{lem}
\begin{proof}
We can suppose $V=\R^d$ with canonical basis $(e_1,\dots)$, $W=\R^k\times\Z^\ell\times\{0\}^{d-k-\ell}$, $\Gamma=\Z^{k+\ell}\times\{0\}^{d-k-\ell}$.

We can pick, for $i\le k+\ell$, $e_i^{(n)}\in W_n$ so that $e_i^{(n)}\to e_i$ for all $i$. Consider the linear map $B_n$ mapping each $e_i$ to $e_i^{(n)}$ (agree $e_i^{(n)}=e_i$ for $i>k+\ell$). Then $B_n\to 1$, so eventually $B_n$ has an inverse $A_n$, and $\Gamma\subset A_nW_n$.
\end{proof}

When $V$ is a finite-dimensional real vector space, there is a canonical identification between $V^\vee$ and the dual $V^*$. Then if $W$ a closed subgroup of $V$, $W^\curlywedge$ corresponds under this identification to $$\{L\in V^*|\forall v\in W,\langle v,L\rangle\in\Z\},$$
which is known, when $W$ is a lattice, as the dual lattice. If $A\in\GL(V)$, then it is immediate that $(AW)^\curlywedge=A^tW^\curlywedge$, where $A^t\in\GL(V^*)$ is the adjoint map of $A$.

\begin{lem}
Let $W$ be a closed subgroup of $V$ and $\Gamma$ a lattice in $W$. Let $W_n$ be a sequence of closed subgroups of $V$ such that $W_n\to W$. Suppose that $W_n$ eventually contains $\Gamma$. Then $W_n\cap W^0\to W^0$ (the unit component of $W$).
\end{lem}
\begin{proof}
Let $w\in W^0$. Then there exists a sequence $w_n\in W_n$ with $w_n\to w$. Set $R=d(W^0,W-W^0)$. We can suppose $d(w_n,W^0)\le 2R/3$ for all $n$. If $w_n\notin W^0$ for infinitely many $n$'s, we can find a positive integer $m_n$ such that $d(m_nw_n,W^0)=m_nd(w_n,W_0)\in[R/3,2R/3]$. Now we can find $\gamma_n\in\Gamma$ such that the sequence $(m_nw_n-\gamma_n)$ is bounded. By assumption it belongs to $W_n$, and $d(m_nw_n-\gamma_n,W^0)\in[R/3,2R/3]$. So it has a cluster value $x$ with $d(x,W^0)\in [R/3,2R/3]$, so $x\notin W$, but $W_n\to W$ forces $x\in W$, a contradiction, i.e. $w_n\in W^0$ for large $n$.
\end{proof}

\begin{prop}There exists a constant $C_d$ with the following property.
Whenever $\Gamma$ is a discrete subgroup of $\R^d$ with shortest vector of length $\ge R$, then $\Gamma^\vee$ is $\eps$-dense, with $\eps=C_d/R$.\label{dide}
\end{prop}
\begin{proof}
We can embed $\Gamma$ in a lattice having the same shortest vector. This reduces to the classical case of lattices, due to Mahler \cite{Mah} (an elementary approach provides $C_d$ with an exponential upper bound with respect to $d$; nevertheless a linear bound on $C_d$ can be obtained \cite{Ban}).
\end{proof}

\begin{proof}[Proof of Theorem \ref{pce}]
We identify $V=\R^d$. Suppose $H_n$ tends to $H=\R^k\times\Z^\ell\times\{0\}^m$ ($k+\l+m=d$) and let us prove that $H_n^\vee$ tends to $\{0\}^k\times\Z^\ell\times\R^m$. Taking subsequences if necessary, we can suppose these sequences converge. Clearly the limit $W$ of $H_n^\vee$ is ``$\R/\Z$-orthogonal" to $H$ and therefore $W$ is contained in $\{0\}^k\times\Z^\ell\times\R^m$.

In view of Lemma \ref{matrices}, we can suppose that $H_n$ contains $\Z^{k+\ell}\times\{0\}^m$. (Indeed, $(A_nH_n)$ tends to $H$ and $(A_nH_n)^\vee=\left({A_n^\top}\right)^{-1}H_n^\vee$ tends to $W$.) We then claim that the orthogonal projection of $H_n$ onto $\{0\}^{k+\ell}\times\R^m$ is a discrete subgroup $\Gamma_n$ of systole (shortest vector) $\to\infty$. Indeed, if $H_n$ contains an element of the form $(x_n,y_n,z_n)$ with $(z_n)$ bounded and nonzero for infinitely many $n$'s, we can first multiply these elements by integers so that $\|z_n\|\ge 1$ for infinitely many $n$'s, and then  by translating by elements of $\Z^{k+\ell}\times\{0\}^m$, which is contained in $H_n$ by assumption, we can suppose $x_n$ and $y_n$ are bounded. Then at the limit, we obtain in $H$ an element whose third coordinate is nonzero, a contradiction. Thus $H_n\subset\R^k\times\Z^\ell\times\Gamma_n$, and the systole of $\Gamma_n$ tends to infinity. So $$H_n^\vee\supset\{0\}^k\times\Z^\ell\times\Gamma_n^\vee.$$ By Proposition \ref{dide}, $\Gamma_n^\vee$ tends to $\R^m$. So $W\supset\{0\}^k\times\Z^\ell\times\R^m$ and we are done.
\end{proof}

\begin{proof}[Proof of Theorem \ref{pc}]
It is explained in \cite{PTd} how Theorem \ref{pce} implies Theorem \ref{pc}; we do not repeat the full argument here, but the reader can complete the proof as follows:
\begin{itemize}
\item Show that the class $\mathcal{C}$ of locally compact abelian groups $G$ for which the orthogonal map $\mathcal{S}(G)\to\mathcal{S}(G^\vee)$ is continuous, is stable under taking
\begin{itemize}
\item
closed subgroups;
\item Pontryagin dual;
\item direct limits (namely if every open compactly generated subgroup of $G$ is in $\mathcal{C}$ then $G$ is in $\mathcal{C}$).
\end{itemize}
\item Check that the smallest class of (isomorphism classes of) locally compact abelian groups containing $\R^d$ for all $d$ and closed under the three operations above is the class of all locally compact abelian groups.\qedhere
\end{itemize}
\end{proof}

\section{Isolated points}

In this section, we prove Theorem \ref{posil}.

\begin{lem}\label{zeri}
The subgroup $\{0\}$ is isolated in $\mathcal{S}(G)$ if and only if $G\simeq T\times D$, with $T\simeq (\R/\Z)^k$ a torus, and $D$ discrete Artinian.
\end{lem}
\begin{proof}
Suppose $\{0\}$ isolated. Then $G$ does not contain any discrete copy of $\Z$, so is elliptic. Let $U$ be a compact open subgroup. Then $\{0\}$ is isolated in $\mathcal{S}(U)$, so $U^\vee$ is isolated in $\mathcal{S}(U^\vee)$. As $U^\vee$ is discrete, this means that $U^\vee$ is finitely generated (see \cite{CGP}), so replacing $U$ by a finite index subgroup, we can suppose that $U$ is a torus. As $U$ is open and divisible, it has a direct factor $D$ in $G$, which is discrete. As $\{0\}$ is isolated in $\mathcal{S}(D)$ as well, we deduce from \cite[Lemma~4.1]{CGP} that $D$ is Artinian. 

Conversely, let us assume that $T$ is a torus, $D$ is Artinian, $G=T\times D$, and let us prove that $\{0\}$ is isolated in $\mathcal{S}(G)$. If $D$ is finite, then the Pontryagin dual is a finitely generated abelian group, so $G^\vee$ is isolated in $\mathcal{S}(G^\vee)$, so $\{0\}$ is isolated in $\mathcal{S}(G)$. In general, if $D_{\text{prime}}$ is the subgroup of $D$ generated by elements of prime order, it is easy to check that $T\times D_{\text{prime}}-\{0\}$ is a discriminating subset of $T\times D$, that is, every non-trivial subgroup of $T\times D$ has non-trivial intersection with $T\times D_{\text{prime}}$. Therefore since $D_{\text{prime}}$ is finite, this reduces to the case when $D$ is finite and we are done.
\end{proof}

We say that a group is {\it adic} if it is isomorphic to a finite direct product $\prod_i\Z_{p_i}$.

\begin{lem}\label{gi}
The subgroup $G$ is isolated in $\mathcal{S}(G)$ if and only if $G\simeq A\times P$, with $A$ a finitely generated abelian group, and $P$ an adic group.
\end{lem}
\begin{proof}
Follows from Lemma \ref{zeri} by Pontryagin-Chabauty duality.
\end{proof}

In the discrete setting, a necessary and sufficient condition for a subgroup $H$ to be isolated in $\mathcal{S}(G)$ is that $H$ be isolated in $\mathcal{S}(H)$ and $\{0\}$ be isolated in $\mathcal{S}(G/H)$. In the locally compact setting, this does not hold any longer, although there is essentially a unique obstruction, given by the following lemma.

\begin{lem}\label{zzrz}
The subgroup $\Z\times\{0\}$ is not isolated in $\mathcal{S}(\Z\times\R/\Z)$.
\end{lem}
\begin{proof}
It is part of the continuous family of subgroups $\langle(1,t)\rangle$ for $t\in\R/\Z$.
\end{proof}

\begin{proof}[{\bf Proof of Theorem \ref{posil}}]
Suppose that $H$ is isolated. So $H$ is isolated in $\mathcal{S}(H)$ and $\{0\}$ is isolated in $\mathcal{S}(G/H)$. By Lemmas \ref{zeri} and \ref{gi}, we deduce that $H\simeq A\times P$, $G/H\simeq T\times D$, with $A$ finitely generated abelian, $P$ adic, $T$ torus, $D$ Artinian. Assume that simultaneously $A$ is infinite and $T$ is non-trivial. Then we can embed $\mathcal{S}(\Z\times\R/\Z)$ into $\mathcal{S}(G)$, mapping $\Z$ to $H$. But $\Z$ is not isolated in $\mathcal{S}(\Z\times\R/\Z)$ by Lemma \ref{zzrz}, so $H$ is not isolated.

Conversely, assume that $H\simeq A\times P$, $G/H\simeq D$, $A$ finitely generated abelian, $P$ adic, $D$ Artinian. Suppose $L$ is close enough to $H$. As $H$ is clopen, this means that $L\cap H$ is close to $H$, and as $H$ is isolated in $\mathcal{S}(H)$ by Lemma \ref{gi}, this implies that $L$ contains $H$. Now using that $\{0\}$ is isolated in $\mathcal{S}(G/H)$, we deduce that $L=H$, so $H$ is isolated. The second case is equivalent by Pontryagin-Chabauty duality.
\end{proof}

\section{Some natural maps}

Let $G$ be a locally compact abelian group. If $\Omega$ is an open subgroup, then the map $\mathcal{S}(G)\to\mathcal{S}(\Omega)$ mapping $H$ to $H\cap\Omega$ is obviously continuous and surjective, and we refer to it as the natural map $i_\Omega:\mathcal{S}(G)\to\mathcal{S}(\Omega)$.

By duality, if $K$ is a compact subgroup, then the projection map $p_K:\mathcal{S}(G)\mapsto\mathcal{S}(G/K)$ mapping $H$ to $(H+K)/K$, is continuous and surjective. (When $\Omega$ is not open or $K$ non-compact, these maps often fail to be continuous, see \cite{PT} for a discussion.)

Note that if $K\subset\Omega$, the maps $i_\Omega$ and $p_K$ commute in an obvious way, and the composition map $\rho_{\Omega,K}:\mathcal{S}(G)\to\mathcal{S}(\Omega/K)$ maps $H$ to $((H+K)\cap\Omega)/K=((H\cap\Omega)+K)/K$.

If $K'\subset K\subset \Omega\subset \Omega'$, then obviously
$$\rho^G_{\Omega,K}=\rho^{\Omega'/K'}_{\Omega/K',K/K'}\circ\rho^G_{\Omega',K'},$$
where we write $\rho^G_{\Omega,K}$ for $\rho_{\Omega,K}$ in order to mention $G$. Therefore, if we have a net $(K_i,\Omega_i)$, filtering in the sense that if $j\ge i$ then 
$K_j\subset K_i\subset \Omega_i\subset \Omega_j$, we obtain a natural map
$$\mathcal{S}(G)\to\underleftarrow{\lim}\mathcal{S}(\Omega_i/K_i)$$
This map has dense image, hence is surjective. It is injective provided $\bigcap K_i=\{0\}$ and $\bigcup\Omega_i=G$, in which case it is a homeomorphism.

We can play another game with these maps. Consider the diagonal map
$$i_\Omega\times p_K:\mathcal{S}(G)\to\mathcal{S}(\Omega)\times\mathcal{S}(G/K).$$
This map need not be surjective. It is interesting because its fibers have a very special form.

\begin{prop}\label{split}
Let $G$ be a locally compact abelian group, $K$ a compact subgroup, $\Omega$ an open subgroup containing $K$. Consider the map
\begin{eqnarray*}
\mathcal{S}(G) & \to & \mathcal{S}(\Omega)\times\mathcal{S}(G/K)\\
 H &\mapsto & (H\cap \Omega, (H+K)/K)
\end{eqnarray*}
as above.
Then the fiber of $(R,M/K)$ is either empty, or homeomorphic to the compact group $$\Hom((M+\Omega)/\Omega,K/(K\cap R))$$ (in particular it is homogeneous).
\end{prop}
\begin{proof}
If $H\cap\Omega=R$ is assumed, then obviously $H+K=M$ is equivalent to $H+K+R=M$. So we can suppose that $R=0$ without changing the statement. Also, we can suppose that $M=G$. So we have to prove that the set $\mathcal{F}$ of subgroups $H$ with $H\cap\Omega=\{0\}$ and $H+K=G$ is either empty or homeomorphic to $\Hom(G/\Omega,K)$. Suppose there is at least one such group $L$. Then since $K\subset\Omega$ we get $G=V\oplus K=V\oplus\Omega$, so $K=\Omega$ (in the original group this means that $\Omega\cap M=K+R$ when the fiber is non-empty) and $V$ is isomorphic to $G/\Omega$. We see that $H\in\mathcal{F}$ if and only if it is the graph of a homomorphism $V\to K$. So $\mathcal{F}$ is homeomorphic to $\Hom(G/\Omega,K)$.
\end{proof}

\section{Groups with $R=0$: study of connected components}\label{sec:rzero}

In this section and the next one, we study connectedness of $\mathcal{S}(G)$. The study of path-connectedness and local connectedness of $G$ itself was done by Dixmier \cite{Dix} and seems to be, to a large extent, fairly unrelated.

\begin{lem}\label{dc}
Let $G$ be a locally compact abelian group. If $G$ is either discrete or compact, then $\mathcal{S}(G)$ is totally disconnected.
\end{lem}
\begin{proof}
If $G$ is discrete, then $\mathcal{S}(G)$ is a closed subset of $2^G$ and is therefore totally disconnected. By duality, we deduce the same result if $G$ is compact.
\end{proof}

\begin{prop}
If $R(G)=0$ and $G$ is infinite, then $\mathcal{S}(G)$ has infinitely many connected components.
\end{prop}
\begin{proof}
If $G$ is discrete, then $\mathcal{S}(G)$ is totally disconnected and infinite and we are done.

Assume $G$ non-discrete. By Lemma \ref{Rop}, there exists a compact open subgroup $M$ in $G$. The obvious map $\mathcal{S}(G)\to\mathcal{S}(M)$ is continuous and surjective. Since $M$ is compact and infinite, its Pontryagin dual is discrete and infinite (see Paragraph \ref{opd}), so $\mathcal{S}(M^\vee)$ is infinite and totally disconnected, and is homeomorphic to $\mathcal{S}(M)$ by Pontryagin-Chabauty duality (Theorem \ref{pc}).
\end{proof}

\begin{lem}\label{ttdd}
Suppose that $G$ is either elliptic, or totally disconnected. Then $\mathcal{S}(G)$ is totally disconnected.
\end{lem}
\begin{proof}
Let $K$ be a compact open subgroup. Proposition \ref{split} provides a continuous map from $\mathcal{S}(G)$ to $\mathcal{S}(K)\times\mathcal{S}(G/K)$, which is totally disconnected by Lemma \ref{dc}. Therefore any connected subset of $\mathcal{S}(G)$ is contained in a fiber of this map, and by Proposition \ref{split} again, any such nonempty fiber is homeomorphic to $\Hom(D,Q)$ for some subgroup $D$ of $G/K$ and some quotient $Q$ of $K$. If $G$ is totally disconnected, so is $Q$, and therefore $\Hom(D,Q)$ is totally disconnected, so $\mathcal{S}(G)$ is totally disconnected. By duality, the same conclusion holds if $G$ is elliptic.
\end{proof}

Let $E_G$ denote the elliptic subgroup of a locally compact abelian group $G$.

\begin{thm}\label{tcon}
Let $G$ be a locally compact abelian group with $R(G)=0$. For any $H\in\mathcal{S}(G)$, set $N=H+G_0$ and $L=H\cap E_G$. Then the connected component of $H$ in $\mathcal{S}(G)$ consists of the subgroups $H'$ such that $H'+M=N$ and $H'\cap E_G=N$, and is homeomorphic to the compact group $$\Hom((N+E_G)/E_G,G_0/(G_0\cap L)).$$
\end{thm}
\begin{proof}
Set $K=G_0$, $\Omega=E_G$. Since $R(G)=0$, $K$ is compact and $\Omega$ is open, so we can apply Proposition \ref{split} to get a map to $\mathcal{E_G}\times\mathcal{S}(G/G_0)$ whose fibers have the desired form. By Lemma \ref{ttdd}, the target space $\mathcal{S}(E_G)\times\mathcal{S}(G/G_0)$ is totally disconnected, and by Lemma \ref{hnz} the fibers are connected.
\end{proof}

\begin{cor}\label{std}
Let $G$ be a locally compact abelian group. Equivalences:
\begin{itemize}
\item $\mathcal{S}(G)$ is totally disconnected;
\item $G$ is either elliptic or totally disconnected.
\end{itemize}
\end{cor}
\begin{proof}
The reverse implication is Lemma \ref{ttdd}. Conversely assume $\mathcal{S}(G)$ is totally disconnected.
It is trivial (and a particular case of several already proved results) that this implies $R(G)=0$. Since $G$ is not elliptic, it contains a closed subgroup $H$ isomorphic to $\Z$. Applying Theorem \ref{tcon}, we obtain that the connected component of $H$ in $\mathcal{S}(G)$ is homeomorphic to $G_0$, therefore is not reduced to a point since $G$ is not totally disconnected.
\end{proof}

\begin{defn}\label{prigid}
A point $x$ of a topological space $X$ is
\begin{itemize}
\item {\it path-rigid} if its path-connected component is reduced to $\{x\}$;
\item {\it rigid} if its connected component is reduced to $\{x\}$.
\end{itemize}
\end{defn}

Plainly, rigid implies path-rigid.
If $\{x\}$ is an intersection of clopen subsets, then $x$ is rigid; the converse holds when $X$ is compact, by \cite[II.4.4, Proposition~6]{BourbakiT}.

Let us denote by $D_{\text{tf}}$ the quotient of the discrete group 
$D$ by its torsion subgroup. 

\begin{cor}
Equivalences:
\begin{itemize}
\item[(i)] $H$ is rigid in $\mathcal{S}(G)$;
\item[(ii)] $H$ is path-rigid in $\mathcal{S}(G)$;
\item[(iii)] $R(G)=0$, and either $H$ is elliptic or $G/H$ is totally disconnected.
\end{itemize}
\end{cor}
\begin{proof}
(i)$\Rightarrow$(ii) is trivial.

(ii)$\Rightarrow$(i) By the theorem, the connected component of $H$ is homeomorphic to a compact group. So we can apply Lemma \ref{ccntp}.

(i)$\Rightarrow$(iii). $R(G)=0$ follows from Proposition \ref{rppr}. Suppose that $H$ is not elliptic and $G/H$ is not totally disconnected. This implies that $N/M$ is not torsion and that $M/L$ is not totally disconnected. By Lemma \ref{hnz}, the connected group $\Hom((N/M)_{\text{tf}},(M/L)_0)$ is non-trivial, and the theorem allows to conclude.

(iii)$\Rightarrow$(i). If $H$ is elliptic, then $N/M$ is torsion; if $G/H$ is totally disconnected, then $M/L$ as well; in both cases 
$\Hom((N/M)_{\text{tf}},(M/L)_0)$ is trivial, and the theorem again allows to conclude.
\end{proof}


We now give a structure result for locally compact abelian groups for which $\pi_0(\mathcal{S}(G))$ is countable, or equivalently scattered. We first recall the following classical proposition, which has reappeared several times in the literature.

\begin{prop}\label{boyy}Let $G$ be a locally compact abelian group.
\begin{itemize}
\item[(1)] If $G$ is discrete, then $\mathcal{S}(G)$ is countable if and only if $G$ lies in an extension
$$1\to Z\to G\to A\to 1,$$
where $Z$ is a finitely generated abelian group, and $A\simeq \mathbf{C}_{m^\infty}$ for some $m\ge 1$.
\item[(2)] If $G$ is compact, then $\mathcal{S}(G)$ is countable if and only if $G$ lies in an extension
$$1\to P\to G\to T\to 1,$$
where $T$ is a compact Lie group and $P\simeq\Z_m$ for some $m\ge 1$. 
\end{itemize}
\end{prop}
Note that in the statement above, $m$ is not assumed prime. As far as we know (1) was first proved by Boyer \cite{Boy} and (2) immediately follows by Pontryagin duality.

\begin{thm}\label{pizero}
Let $G$ be a compact-by-discrete locally compact abelian group. Equivalences:
\begin{itemize}
\item[(a)] $\pi_0(\mathcal{S}(G))$ is scattered.
\item[(b)] $\pi_0(\mathcal{S}(G))$ is countable. 
\item[(c)] $G$ has a compact open subgroup $M$ such that $\mathcal{S}(M)$ and $\mathcal{S}(G/M)$ are countable, and one of the following condition holds
\begin{itemize}
\item[(1)] $M$ is finite;
\item[(2)] $G/M$ is finite;
\item[(3)] $M$ is profinite and $G/M$ is torsion;
\item[(4)] $M$ is virtually connected and $G/M$ has finite torsion. 
\end{itemize}
\item[(d)] One of the following condition holds
\begin{itemize}
\item[(1')] $G$ is discrete with countably many subgroups (see Proposition \ref{boyy});
\item[(2')] $G$ is compact with countably many subgroups (dual to the previous case);
\item[(3')] $G=\Q_\ell\times\Z_m\times \mathbf{C}_{n^\infty}\times F$, where $F$ is finite, $\ell,m,n\ge 1$, $\ell$ is prime with $mn$.
\item[(4')] $G=K\times D\times F$ with $D$ torsion-free discrete, $K$ compact  connected, $F$ finite, and $\mathcal{S}(D)$ and $\mathcal{S}(K)$ are countable.
\end{itemize}
\end{itemize}
\end{thm}
Note that in Cases (1),(2),(3), $\mathcal{S}(G)$ is countable itself. In (4), $\mathcal{S}(G)$ is not totally disconnected unless $G$ is compact or discrete.

We need some preliminary lemmas. 

\begin{lem}\label{caradic}
Let $G$ be a locally compact abelian group with a compact open subgroup $M$. Suppose that $G$ is divisible and torsion-free. Assume that $M$ is adic and $G/M$ is discrete artinian. Then $G$ is local.
\end{lem}
\begin{proof}
Let $S$ be the (finite) set of primes occurring in the canonical decomposition of $M$. Since this decomposition is canonical, it does not depend on $M$ and it follows that $G$ also inherits such a decomposition. Hence we can reduce to the case when $S=\{p\}$. So $M$ is isomorphic to $\Z_p^m$ for some $m$. So we have a continuous embedding $q:M\to\Q_p^m$. Denote by $i$ the inclusion of $M$ into $G$. Since $G$ is divisible as an abstract abelian group, there exists an abstract homomorphism $f:\Q_p^m$ such that $f\circ q=i$. Since the restriction of $f$ to the open subgroup $q(M)=\Z_p^m$ is continuous, $f$ is continuous as well. The kernel of $f$ has trivial intersection with $\Z_p^m$, hence is trivial, that is, $f$ is injective. Besides, the image $f(\Q_p^m)$ contains $M$, hence is open, and is divisible. So it has a direct factor $\Gamma$ in $G$, necessarily discrete. Now $\Gamma$ is torsion-free as a subgroup of $G$, and artinian as a subgroup of $G/M$. So $\Gamma=\{0\}$.
\end{proof}

\begin{lem}\label{caraalocal}
Let $G$ be a locally compact abelian group with a compact open subgroup $M$. Assume that $M$ is adic and $G/M$ is discrete artinian. Then we can write $G$ as $Q\times Z\times D\times F$ with $Q$ local, $Z$ torsion-free adic, $D$ discrete artinian, and $F$ finite.
\end{lem}
\begin{proof}
We can suppose that $M$ is torsion-free. Let $T$ be the torsion subgroup in $G$. Since $T\cap M=\{0\}$, we know that $T$ is a discrete subgroup. Moreover $T$ embeds into $G/M$, so $T$ is Artinian. Let $S$ be the divisible part of $T$, which has finite index in $T$. Since $T_d\cap M=\{0\}$, $D$ has a direct factor $G_1$ in $G$ containing $M$. Working similarly in the Pontryagin dual of $G_1$, we write $G=Z\times D\times G_2$ where $Z$ is torsion-free adic, and both $G_2$ and its Pontryagin dual have finite torsion. Since the Pontryagin dual of $G_2$ has finite torsion, the index $[G_2:nG_2]$ is bounded independently of $n$. So the divisible part $\bigcap_n n!G_2$ of $G_2$ has finite index and we can write $G_2=F\times G_3$ with $F$ finite, $G_3$ divisible and torsion-free (and satisfying the hypotheses of the lemma). We have $G=Z\times D\times F\times G_3$. By Lemma \ref{caradic}, $G_3$ is local.
\end{proof}

\begin{proof}[Proof of Theorem \ref{pizero}]
\begin{itemize}
\item (b)$\Rightarrow$(a) is clear.

\item (c)$\Rightarrow$(b). First note that the result is clear if (1)  or (2) holds, since then $\mathcal{S}(G)$ is countable.

Consider the natural map $\mathcal{S}(G)\to\mathcal{S}(M)\times\mathcal{S}(G/M)$. By Proposition \ref{split}, each non-empty fiber is of the form $\Hom(N,Q)$ with $N\le G/M$ and $Q\le M$. If (3) holds, then $Q$ has finite torsion $Q_t$ and the divisible part $N_d$ of $N$ has finite index, so $\Hom(N,Q)=\Hom(N/N_d,Q_t)$ is finite. So the fibers are finite and $\mathcal{S}(G)$ is countable.

\item (a)$\Rightarrow$(c). Suppose that $\pi_0(\mathcal{S}(G))$ is scattered. Then $\pi_0(\mathcal{S}(M))$ and $\pi_0(\mathcal{S}(G/M))$ are scattered as well, hence countable (Lemma \ref{dc}). Moreover, for every fiber $\mathcal{F}$ of the natural map to $\mathcal{S}(M)\times\mathcal{S}(G/M)$, we have $\pi_0(\mathcal{F})$ scattered. By Proposition \ref{split}, these fibers are homeomorphic to compact groups, so they have finitely many components. Therefore for every closed subgroups $0\le L\le M\le N\le G$, we have $\Hom(N/M,M/L)$ virtually connected. 

\begin{itemize}\item Suppose $\Z$ embeds into $G/M$. Then taking $N$ such that $N/M\simeq \Z$, we have $\Hom(N/M,M)=M$ virtually connected.
\item In a dual way, if $M$ is not profinite, that is, if $\Z$ does not embed into $M^\vee$, then $(G/M)^\vee$ is virtually connected, that is, $G/M$ has finite torsion.  
\end{itemize}
Therefore one of the following holds
\begin{itemize}
\item $M$ is finite (Case (1))
\item $M$ is infinite. \begin{itemize} \item $G/M$ is not torsion. Then $M$ is virtually connected; as it is infinite, it is not profinite, and hence $G/M$ has finite torsion. This is Case (4).\item $G/M$ is torsion. We again discuss.\begin{itemize}\item $M^\vee$ is not torsion. Then $(G/M)^\vee$ is virtually connected, so $G/M$ has finite torsion, so is finite (Case (2)).\item $M^\vee$ is torsion, so $M$ is profinite (Case (3)).\end{itemize}\end{itemize}
\end{itemize}

\item (c)$\Rightarrow$(d). Under the assumption that $\mathcal{S}(M)$ and $\mathcal{S}(G/M)$ are countable, let us prove (i)$\Rightarrow$(i') for $i=1\dots 4$. For $i=1,2$ there is nothing to prove. 

Suppose (4). We can suppose that $M$ is connected. Hence it is divisible, so has a direct factor $H$ in $G$. Then $H$ is discrete and finite-by-(torsion-free). So $H^\vee$ is connected-by-finite, and again the connected component, by the same argument, has a direct factor. This means that $H$ is the direct product of a finite group by a torsion-free group. 

Suppose (3). Since $G/M$ is discrete torsion and $\mathcal{S}(G/M)$ is countable, $G/M$ is artinian of the form $\mathbf{C}_{n^\infty}\times F_1$ with $F_1$ finite and $n$ square-free \cite{Boy}. Dually, $G/M$ is adic of the form $\Z_m\times F_2$ with $F_2$ finite and $m$ square-free. By Lemma \ref{caraalocal}, we can write $G=Q\times Z\times D\times F$ with $Q$ local, $Z$ torsion-free adic, $D$ discrete artinian, and $F$ finite. Because of the special form of $M$ and $G/M$, We can write $Q$, $Z$ and $D$ in the desired form.\qedhere 
\end{itemize}\end{proof}

\section{Connectedness when $R\ge 1$}\label{seco}

\begin{lem}\label{pathbase}
Consider a group $G=\R^k\times H$ with $k\ge 1$, where the locally compact group $H$ has a compact open subgroup (but need not be abelian). Let $L$ be a closed subgroup of $G$. Denote by $L_1$ the closure of the projection of $L$ on $H$. We consider the automorphism $\tau_\lambda(x,h)=(\lambda x,h)$. Consider the path $\tau_{\lambda^{-1}}(L)$ for $\lambda\in[1,+\infty[$. Then for some vector subspace $W$ of $\R^k$, $\tau_{\lambda^{-1}}(L)$ tends to $W\times L_1$ when $\lambda\to +\infty$. 
\end{lem}
\begin{proof}
Let $M$ denote any accumulation point of $(\tau_{\lambda^{-1}}(L))$. It is straightforward that $\{0\}\times L_1\subset M\subset\R\times L_1$.

Let us begin by the case when $H$ is compact. Let $W$ be the vector space generated by the projection of $L$ on $\R^k$. Obviously $M\subset W\times L_1$. Take $v\in W$ and $\lambda_i\to +\infty$. Then $\lambda_iv$ is at bounded distance, say $\le r$, of some element in the projection of $L$. That is, we can find $(w_i,h_i)$ in $L$ with $\|w_i-\lambda_iv\|\le r$. So $\|\lambda_i^{-1}w_i-v\|\to 0$. So $\tau_{\lambda_i^{-1}}(w_i,h_i)$ accumulates to $(v,h)$ for some $h\in L_1$, by compactness of $H$. Therefore, since we know that $\{0\}\times L_1$ is contained in $M$, we deduce that $(v,0)$ is contained in $M$. Since $v$ is arbitrary, $M=W\times L_1$.

In general, $H$ has a compact open subgroup $K$, and by the above, if $W$ is the vector space generated by the projection of $L\cap (\R^k\times K)$ on $\R^k$, then we obtain that $\tau_{\lambda^{-1}}(L)\cap (\R^k\times K)\to W\times (L_1\cap K)$. As $M$ has to be of the form $F\times L_1$ for some closed subgroup $F$, we deduce that $F=W$ and 
$\tau_{\lambda^{-1}}(L)$ tends to $W\times L_1$.
\end{proof}

\begin{prop}\label{rppr}
Let $G=\R^k\times H$ be a locally compact group, where $H$ has a compact open subgroup. Let $L$ be a closed subgroup of $G$ and let $L_1$ be the closure of the projection of $L$ on $H$.
Then the path-connected component of $L$ in $\mathcal{S}(G)$ contains $\mathcal{S}(\R^k)\times \{L_1\}$. In particular, if $k\ge 1$, then $L$ is not path-rigid (see Definition \ref{prigid}).\label{pba}
\end{prop}
\begin{proof}
This follows from Lemma \ref{pathbase}, and path-connectedness of $\mathcal{S}(\R^k)$ (which is an easy exercise).
\end{proof}

Consider a group $G=\R^k\times H$ and $L$ a closed subgroup of $G$. Denote by $M$ the closure of the projection of $L$ on $H$. We consider the automorphism $\tau_\lambda(x,h)=(\lambda x,h)$. Consider the path $\tau_{\lambda^{-1}}(L)$ for $\lambda\in[1,+\infty[$.

Define a locally compact abelian group as {\it circular} if it is discrete and has an injective homomorphism into $\R/\Z$ (for instance, $\R/\Z$ with the discrete topology is circular). Define a locally compact abelian group 
$G$ to be {\it polycircular} if it has a composition series
$$\{0\}=G_0\subset G_1\subset\dots\subset G_n=G$$ such that each $G_i/G_{i-1}$ is circular, and {\it metacircular} if it can be written as $\R^k\times H$, where $H$ has a compact open subgroup $M$ such that both $H/M$ and $M^\vee$ are polycircular. For instance, every compactly generated Lie group is metacircular. Also, $\Q_p$ is metacircular, but the infinite direct product $(\Z/2\Z)^\Z$ is not.

\begin{lem}\label{boutdechemin}
Let $H$ denote a locally compact abelian group such that either $H$ or $H^\vee$ is circular. 
In $\mathcal{S}(\R^k\times H)$, if $k\ge 1$, we can join $\{0\}\times H$ to $\{0\}$ by a path.
\end{lem}
\begin{proof}It is enough to prove the lemma for $k=1$, since then we can join $\R^{\ell}\times H$ to $\R^{\ell-1}\times H$ by a path and concatenate all those paths.
\begin{itemize}
\item $H$ has a continuous injection $\psi$ to $\R/\Z$. Consider in $\R/\Z\times H$ the graph of $\psi$ (upside down). Let $L$ be its inverse image in $\R\times H$. On $G=\R\times H$, consider the automorphism $\tau_\lambda(x,h)=(\lambda x,h)$. 

Consider the path $\tau_{\lambda}(L)$ for $\lambda\in]0,+\infty[$. 
Clearly, when $\lambda\to 0$, it tends to $\R\times H$.
When $\lambda$ tends to $+\infty$, we claim that it tends to $\{0\}$. 
Indeed consider $(\lambda_i)\to\infty$, and $(t_i,h_i)\to (t,h)$ with $(t_i,h_i)\in\tau_{\lambda_i}(L)$. This means that $(\lambda_i^{-1}t_i,h_i)\in L$. So $\psi(h_i)=\lambda_i^{-1}t_i$ in $\R/\Z$. 
As $(t_i,h_i)\to (t,h)$ and $\lambda_i^-1\to 0$, we have $\psi(h_i)\to\psi(h)$ and $\lambda_i^{-1}t_i\to 0$. So $\psi(h)=0$ in $\R/\Z$. By injectivity of $\psi$, we deduce that $h=0$.
We obtain that in restriction to the clopen subset $\R\times (H-\{0\})$, the path $\tau_{\lambda}(L)$ tends to $\emptyset$ when $\lambda$ tends to $+\infty$. In restriction to $\R\times\{0\}$, this is constantly $\{0\}$. So the claim is proved.

\item $H^\vee$ is discrete and injects into $\R/\Z$. By Pontryagin-Chabauty duality, it amounts to join $\R\times\{0\}$ and $\R\times H^\vee$ in $\mathcal{S}(\R\times H^\vee)$. Using connectedness of $\mathcal{S}(\R)$, we can join $\R\times\{0\}$ and $\{(0,0)\}$, respectively $\{0\}\times H^\vee$ and $\R\times H^\vee$, so it is enough to join $\{(0,0)\}$ and $\{0\}\times H^\vee$, which was done in the previous case.\qedhere
\end{itemize}
\end{proof}

\begin{thm}\label{con}
If $R(G)\ge 1$ then $\mathcal{S}(G)$ is connected. If moreover $G=\R^k\times M$ and $M$ is metacircular (e.g.~if $G$ is a compactly generated Lie group), then $\mathcal{S}(G)$ is path-connected.
\end{thm}
\begin{proof}
First assume that $M$ is metacircular. Set $G=\R\times H$. 

In view of Proposition \ref{pba}, it is enough to show that every closed subgroup of the form $\{0\}\times M$, can be joined to $\{0\}$ by a path. We can find a composition series
$$\{0\}=M_0\le\dots\le M_n=M$$
such that each $M_n/M_{n+1}$ is isomorphic to $\Z/d\Z$ for some $d=d(n)\ge 1$, $\Z$ or $\R/\Z$. 
From Lemma \ref{boutdechemin}, we can join $\R\times M_i$ to $\R\times M_{i-1}$ by a path. Combining, we join $\R\times M$ and $\R\times\{0\}$. This proves the second statement.

The first statement is deduced by the projective limit argument. First assume that $G$ is a Lie group. Then $M$ is the projective limit of its open compactly generated subgroups $M_i$, so $\mathcal{S}(G)$ is the projective limit of $\mathcal{S}(\R^k\times G_i)$, which are compact and connected. So $\mathcal{S}(G)$ is connected.

By Pontryagin duality, we deduce that $\mathcal{S}(G)$ is connected when $G$ is any compactly generated locally compact abelian group with $R(G)\ge 1$ (see Paragraph \ref{opd}). We can reiterate a second time the projective limit argument to deduce that $\mathcal{S}(G)$ is connected for any locally compact abelian group with $R\ge 1$. 
\end{proof}

From Theorem \ref{con}, we see that $\mathcal{S}(G)$ can be path-connected even when $G$ is not a compactly generated Lie group, for instance when $G=\R\times D$ with $D$ a discrete group with a injection into $\R/\Z$. However, path-connectedness may fail for some groups with $R\ge 1$, as the following example shows.

\begin{prop}Let $A$ be an infinite, discrete abelian group of uniform torsion. Then $\mathcal{S}(\R\times A)$ is not path-connected.\label{sra}
\end{prop}
\begin{proof}
Consider the natural map $\rho:\mathcal{S}(\R\times A)\to\mathcal{S}(\R)$. Identifying $\lambda\Z$ with $\lambda^{-1}$, $\{0\}$ with $0$ and $\R$ with $\infty$, we can view $\rho$ as a continuous map to $[0,\infty]$.

In restriction to $\rho^{-1}(]0,+\infty])$, an easy compactness argument shows that the projection map $w$ to $\mathcal{S}(A)$ is continuous. 

Besides, in restriction to $\rho^{-1}([0,+\infty[)$, the intersection map $w'$ to $\mathcal{S}(A)$ is continuous. Indeed if $(H_i)$ tends to $H$, a priori $H_i\cap A$ could tend to something smaller than $H\cap A$. But if $(0,h)\in H\cap A$, it is approximable by elements $(e_i,h_i)$ of $H_i$; so $(de_i,0)$ belongs to $H_i$ (where $dA=0$); unless $e_i$ is eventually zero, this would imply that $\rho(H_i)\to\infty$, contradiction.

We claim that if $A_0$ and $A_1$ are subgroups of $A$, and if $\R\times A_0$ and $\R\times A_1$ can be joined by a path, then $A_0$ and $A_1$ are commensurable (the converse is an easy consequence of Theorem \ref{con}). Consider a path $\gamma$ in $\mathcal{S}(\R\times A)$ with $\gamma(i)=\R\times A_i$ ($i=0,1$). By compactness, we can find $0=t_0<\dots<t_1<\dots <t_k=1$ such that each $\gamma([t_i,t_{i+1}])$ is contained in either $\rho^{-1}(]0,+\infty])$ or $\rho^{-1}([0,+\infty[)$. In the first case, $w(\gamma(t_i))=w(\gamma(t_{i+1})$ by continuity and connectedness, and because $\mathcal{S}(A)$ is totally disconnected. Similarly, in the second case, $w'(\gamma(t_i))=w'(\gamma(t_{i+1})$. But $w'(H)$ is a finite index subgroup of $w(H)$ for any $H$, since any subquotient of $\R$ of uniform torsion is finite. Therefore in all cases, $w(\gamma(t_i))$ is commensurable to $w(\gamma(t_{i+1}))$. Accordingly $w(t_0)=A_0$ and $w(t_k)=A_1$ are commensurable.
\end{proof}

\section{Dimension}\label{sec:Dimension}
In this section, we have to switch from inductive to topological dimension and vice versa (see the reminder in Paragraph \ref{todi}), depending on the tools available.

The following lemma is Theorem VI.7 in \cite{HW}.
\begin{lem}\label{dfiber}
Let $f$ be a closed map between separable metrizable spaces. If all fibers of $f$ have dimension $\le b$, then $\inddim(X)\le\inddim(Y)+b$.
\end{lem}

The following theorem is a corollary of Kloeckner's stratification \cite{Klo} of $\mathcal{S}(\R^d)$.

\begin{thm}\label{rddd}
$$\inddim(\mathcal{S}(\R^d))=d^2.$$
\end{thm}

\begin{proof}Since the set of lattices is an open $d^2$-dimensional manifold in $\mathcal{S}(\R^d)$, it is clear that $\inddim(\mathcal{S}(\R^d))\ge d^2$. Also, $\mathcal{S}(\R^d)$ has a natural finite partition into $\GL_d(\R)$-orbits, each of which is a manifold of dimension $\le d^2$. However, it is not clear that this directly provides the desired inequality $\inddim(\mathcal{S}(\R^d))\le d^2$. In \cite{Klo}, Kloeckner proves that the above partition is a {\it Goresky-MacPherson stratification}, and it directly follows from the definition (which we do not recall here) that every point in a $n$-dimensional Goresky-MacPherson stratified space has a system of neighbourhoods whose boundaries are $(n-1)$-dimensional Goresky-MacPherson stratified spaces; in particular the inductive dimension of an $n$-dimensional Goresky-MacPherson stratified space is $\le n$.
\end{proof}

\begin{defn}
Let $X$ be a locally compact space. Define $\inddim_\infty(X)$ as the dimension of the Alexandrov compactification of $X$ at the point $\infty$, that is,
$$\inddim_\infty(X)=\sup_K(1+\inf_L\{\inddim \partial L:L\supset K\}),$$
where $K,L$ range over compact subsets of $G$.
\end{defn}

\begin{lem}\label{diminf}
Let $G$ be a metrizable locally compact abelian group. Then $\inddim_\infty(G)=0$ if $G$ is compact-by-discrete, and $\inddim_\infty(G)=\inddim(G)$ otherwise.
\end{lem}
\begin{proof}
The statement remains the same if we replace $G$ by an open subgroup. Hence we can suppose $G=\R^k\times K$ with $K$ compact. Then the point at infinity has neighbourhoods with boundary of the form $S\times K$, where $S$ is a $(k-1)$-sphere. Observe that $S\times K$ has dimension $\le (k-1)+\inddim(K)$ if $k\ge 1$ and is empty if $k=0$. So if $k\ge 1$, we get $\inddim_\infty(G)\le k+\inddim(K)=\inddim(G)$, the latter equality using Lemma \ref{covg}, and if $k=0$ we get $\inddim_\infty(G)=0$.

Conversely, denoting by $D^n$ the closed $n$-disc, we use the fact that as a consequence of Lemma \ref{covg}, $G$ contains a closed subset homeomorphic to $\R^k\times D^n$ for $n=\inddim(K)$ (or for any $n$ if $\inddim(K)=\infty$). So if $k\ge 1$, the Alexandrov compactification of $G$ contains the Alexandrov compactification of $\R_{\ge 0}\times D^{n+k-1}$, which is a $(k+n)$-sphere passing through the point at infinity. So $\inddim_\infty(G)\ge\inddim(G)$. 
\end{proof}

\begin{lem}\label{dgz}
Let $G$ be a metrizable locally compact abelian group. Then
$$\inddim(\mathcal{S}(G\times\Z))\le\inddim(\mathcal{S}(G))+\inddim(G).$$
\end{lem}
\begin{proof}
Consider the natural map $\mathcal{S}(G\times\Z)\to\mathcal{S}(G)$. The fiber of $H$ is homeomorphic to the set of partially defined homomorphisms $\Z\to G/H$, which is homeomorphic to the Alexandrov compactification of $G/H\times\Z_{>0}$. A point not at infinity in this space has dimension $\inddim(G/H)$. The point at infinity has a basis of neighbourhood given by the complements of $K\times F$, for $K$ compact in $G/H$ and $F$ finite in $\Z_{>0}$, so if we restrict to $K$ with $\inddim\partial K\le\inddim_\infty(G/H)-1$, we deduce that the dimension of the fiber is $\le\max(\inddim(G/H),\inddim_\infty(G/H))$. By Lemma \ref{dfiber}, we get 
$$\inddim(\mathcal{S}(G\times\Z))\le\inddim(\mathcal{S}(G))+\sup_{H\in\mathcal{S}(G)}(\max(\inddim(G/H),\inddim_\infty(G/H))).$$
Now by Lemma \ref{diminf}, $\inddim_\infty(G/H)\le\inddim(G/H)$, and by Lemma \ref{covg}, $$\inddim(G/H)\le\inddim(G)\qedhere.$$
\end{proof}

\begin{thm}\label{klm}
The inductive, or topological, dimension of $\mathcal{S}(\R^k\times\Z^\ell\times(\R/\Z)^m\times F)$, for $F$ finite, is $(k+\ell)(k+m)$, and is achieved by a piece of manifold.
\end{thm}
\begin{proof}Since $\mathcal{S}(\R^k\times\Z^\ell\times(\R/\Z)^m\times F)$ is compact and metrizable, the two notions of dimension coincide (Proposition \ref{toporef}(\ref{ury})) so we can work with the inductive dimension.

First, we have to embed a $(k+\ell)(k+m)$-dimensional manifold into $\mathcal{S}(\R^k\times\Z^\ell\times(\R/\Z)^m)$. It will be convenient to rewrite the group as $\Z^\ell\times\R^k\times(\R/\Z)^m$. 
We consider the action of the group of automorphisms of the form $(x,y,z)\mapsto (x,Ay+Bx,z+Cx+Dy)$, where $A\in\GL(\R^k)$, $B\in\Hom(\Z^\ell,\R^k)\simeq \R^{\ell k}$, $C\in\Hom(\Z^\ell,(\R/\Z)^m)\simeq(\R/\Z)^{\ell m}$, $D\in\Hom(\R^k,(\R/\Z)^m)\simeq\R^{km}$.
Consider the subgroup $\Z^{k+\ell}$: its stabilizer is discrete, so its orbit is $(k+\ell)(k+m)$-dimensional; this is the desired piece of manifold and provides the easy inequality.

Conversely, to obtain that the inductive dimension is bounded as given, we first reduce the case from $G\times F$ to $G$. Each fiber of the map $\mathcal{S}(G\times F)\to\mathcal{S}(G)$ can be identified to the set of partial homomorphisms from $F$ to some quotient of $G$, hence is finite. So by Lemma \ref{dfiber}, we have $\inddim(\mathcal{S}(G\times F))\le\inddim(\mathcal{S}(G))$, the other inequality being trivial. 

By Lemma \ref{dgz}, we obtain by induction that $$\inddim(\mathcal{S}(\R^k\times\Z^\ell))\le\inddim(\mathcal{S}(\R^k))+kl.$$
By duality, 
$$\inddim(\mathcal{S}(\R^k\times(\R/\Z)^m))\le\inddim(\mathcal{S}(\R^k))+km,$$
and by a second induction,
$$\inddim(\mathcal{S}(\R^k\times\Z^\ell\times(\R/\Z)^m))\le\inddim(\mathcal{S}(\R^k))+km+\ell(k+m),$$
and finally by Theorem \ref{rddd}
$$\inddim(\mathcal{S}(\R^k\times\Z^\ell\times(\R/\Z)^m))\le (k+\ell)(k+m).$$
\end{proof}

We can now state the general result (see Lemma \ref{covg} and Corollary \ref{dimgvc} for interpretations of $\topdim(G)$ and $\topdim(G^\vee)$).

\begin{thm}\label{covdim}
Let $G$ be a locally compact abelian group. The topological dimension of $\mathcal{S}(G)$ is given by $$\topdim(\mathcal{S}(G))=\topdim(G)\topdim(G^\vee),$$ where $0\infty=\infty 0=0$; in case this value is finite, it is achieved by a piece of manifold.
\end{thm}
\begin{proof}
For the inequality $\ge$, Lemma \ref{ds} reduces to $G=\Z^{(I)}\times(\R/\Z)^J\times\R^k$. If $I,J$ are finite, then by Theorem \ref{klm} we get a piece of $n$-manifold, where $n=(k+\# I)(k+\# J)$. If $(k+\# I)(k+\# J)\neq 0$ and either $I$ or $J$ is infinite, the same argument allows to find pieces of manifolds of arbitrary large dimension. 

Let us prove $\le$. If either $\topdim(G)$ or $\topdim(G^\vee)$ is zero, then $G$ is totally disconnected, or elliptic, and then we know by Corollary \ref{std} that $\mathcal{S}(G)$ is totally disconnected, so is zero-dimensional.

We henceforth assume $\topdim(G)\topdim(G^\vee)$ nonzero. If either $\topdim(G)$ or $\topdim(G^\vee)$ is infinite there is nothing to prove, so we suppose both finite and nonzero.

In case $G$ is a compactly generated Lie group, in view of Proposition \ref{rcd}, the result is given by Theorem \ref{klm}.

First assume that $G$ is a Lie group. Then in view of Corollary \ref{dimgvc}, $G$ has an open, compactly generated subgroup $M$ such that for every subgroup containg $N$, we have $\topdim(N^\vee)=\topdim(G^\vee)$. Now $\mathcal{S}(G)$ is the projective limit of $\mathcal{S}(N)$, when $N$ ranges over subgroups containg $M$, which is of dimension $\topdim(G)\topdim(G^\vee)$. By Proposition \ref{toporef}(\ref{invers}), it follows that $\topdim(\mathcal{S}(G))\le \topdim(G)\topdim(G^\vee)$.

Now by duality, the result holds when $G$ is a compactly generated locally compact abelian group. Repeating a second time the projective limit argument, we obtain the result for a general locally compact abelian group.
\end{proof}


\end{document}